\documentclass[11pt]{amsart}

\usepackage[pagewise]{lineno}
\usepackage{amsmath,amssymb,graphicx}
\usepackage{amssymb}
\usepackage{bm}
\usepackage[centertags]{amsmath}
\usepackage{amsfonts}
\usepackage{amsthm}
\usepackage{mathrsfs}
\usepackage[usenames]{xcolor}
\usepackage[normalem]{ulem}
\usepackage{comment}

\usepackage[bookmarks=true,hyperindex,pdftex,colorlinks, citecolor=blue,linkcolor=blue, urlcolor=blue]{hyperref}
 
\usepackage{amsfonts}

\usepackage{amscd, amsbsy, upref}
\usepackage{mathtools}
\usepackage{color}
\usepackage[T1]{fontenc} 
\usepackage{url}
\usepackage{footnote}

\makeatletter
\renewcommand*\env@matrix[1][\arraystretch]{%
  \edef\arraystretch{#1}%
  \hskip -\arraycolsep
  \let\@ifnextchar\new@ifnextchar
  \array{*\c@MaxMatrixCols c}}
\makeatother

\numberwithin{equation}{section}

\newcounter{fakecnt}[subsection]

\numberwithin{equation}{section}

\newtheorem{thm}{\bf Theorem}[section]

\theoremstyle{definition}

\newtheorem{cor}[thm]{\bf Corollary} 
\newtheorem{prop}[thm]{\bf Proposition}

\theoremstyle{definition}

\theoremstyle{definition}
\newtheorem{dfn}[thm]{\bf Definition}

\newtheorem{rmk}[thm]{\bf Remark}

\newtheorem{ex}[thm]{\bf Example}

\newtheorem*{theorem*}{Theorem}{\bf}{\it}
\newtheorem*{proposition*}{Proposition}{\bf}{\it}
\newtheorem*{observation*}{Observation}{\bf}{\it}
\newtheorem*{lemma*}{Lemma}{\bf}{\it}

\newtheorem*{Mthm*}{\bf Main Theorem}

\theoremstyle{definition}

\newtheorem*{cor*}{Corollary}

\newtheorem*{prp*}{Proposition}

\newtheorem*{rmk*}{Remark}

\newtheorem{dfn*}{\bf Definition}[section]

\theoremstyle{definition}

\theoremstyle{remark}
\newtheorem{remark}[thm]{\bf Remark}

\def\XXint#1#2#3{{\setbox0=\hbox{$#1{#2#3}{\int}$ }
\vcenter{\hbox{$#2#3$ }}\kern-.6\wd0}}

\DeclareMathOperator{\diam}{diam}
\DeclareMathOperator{\conv}{\textrm{conv}}

\DeclareMathOperator{\dist}{\textrm{dist}}

\newcommand{\vertiii}[1]{{\left\vert\kern-0.35ex\left\vert\kern-0.35ex\left\vert #1 \right\vert\kern-0.35ex\right\vert\kern-0.35ex\right\vert}}

\newcommand{\coo}{\mathrm{c}_{00}}
\newcommand{\co}{\mathrm{c}_0}

\begin{document}
\title[]{Remarks on the FPP in Banach spaces with unconditional Schauder basis}
\thanks{The research was partially supported by FUNCAP/CNPq/PRONEX Grant 00068.01.00/15.\\
{\it Date: 22.07.2023}}

\author{Cleon S. Barroso}
\address{Department of Mathematics, Federal University of  Cear\'a, Cear\'a, Fortaleza 60455-360, Brazil}
\email{cleonbar@mat.ufc.br}
 

\begin{abstract}
This paper brings new results on the FPP in Banach spaces $X$ with a Schauder basis. We first deal with the problem of whether there is a Banach space isomorphic to $\co$ having the FPP. We show that the answer is negative if $X$ contains a pre-monotone basic sequence equivalent to the unit basis of $\co$. We then study sufficient conditions to ensure the existence of such sequences. Interesting results are obtained, including the case when $X$ has a $1$-suppression unconditional basis and its unit ball fails the PCP. With regarding the weak-FPP, we establish two fixed-point results. First, we show that under certain conditions this property is invariant under Banach-Mazur distance one. Next, it is shown that when the basis is either $1$-suppression unconditional or $1$-spreading then $X$ has the weak-FPP provided that a Rosenthal's type property on block basis is verified.
\end{abstract}

\maketitle


\section{Introduction}\label{sec:1}
A closed bounded convex subset $C$ of a Banach space $X$ is said to have the {\it fixed point property} (FPP for short) when every nonexpansive (i.e., $1$-Lipschitz continuous) mapping $T\colon C\to C$ has a fixed point. We say that $X$ has the FPP if every such $C$ does. The FPP emerged in the 1960s, and has since been a highly active research area, highlighting not only its relevance per se, but also its links with Banach space theory. We refer the works of F. Browder \cite{Brd1,Brd2} and W. Kirk \cite{Ki} who obtained the first positive results in Hilbert, uniformly convex or, more generally, reflexive spaces with normal structure. In 1981 B. Maurey \cite{M} proved that reflexive subspaces of $L_1[0,1]$ have FPP. In contrast, D. Alspach \cite{A} exhibited a weakly compact convex set in $L_1[0,1]$ failing it. Noteworthily, P. Dowling and C. Lennard \cite{DLT2} proved that nonreflexive subspaces of $L_1[0,1]$ fail FPP. One can go through \cite{BP} for recent studies on FPP in generalized Lebesgue spaces. Many of the issues arising in the study of the FPP are in fact intrinsic manifestations of the geometric nature of norms. It is then interesting to know which renormings of $X$ have FPP. Recall that $\co$ and $\ell_1$ are standard examples of spaces for which FPP fails. Additionally, $\ell_1(\Gamma)$ ($\Gamma$ uncountable) and $\ell_\infty$ are examples of spaces that cannot be equivalently renormed to have FPP \cite{DLT1, GK}. Hern\'andez-Linares \cite{H-L} shows that $L_1[0,1]$ can be renormed to have the FPP for affine nonexpansive maps. In \cite{DJLT, DLT2, DLT3} the authors consider the impact of asymptotically isometric copies of $\ell_1$ and $\co$ on the failure of FPP. However, in \cite[Theorem 1]{DJLT} it is shown that the family $\{ \| \cdot\|_\gamma\}_{\gamma}$ of renormings of $\ell_1$ given by $\| x\|_\gamma=\sup_{n\in \mathbb{N}}\gamma_n \sum_{k=n}^\infty | \xi_k|$, $x=(\xi_n)_n \in \ell_1$, where $\gamma=\{ \gamma_n\}_{n=1}^\infty\subset (0,1)$ strictly increases to $1$, fail to contain asymptotically isometric copies of $\ell_1$. Surprisingly, in 2008 P. K. Lin \cite{Lin2} proved that $(\ell_1, \| \cdot\|_{\gamma_\star})$ has the FPP for $\gamma_\star=\{ \frac{8^n}{1 + 8^n}\}_{n=1}^\infty$. This became the first example of a nonreflexive space with the FPP. In addition to Lin's result, it is worth pointing out a 2009's result of Domingu\'ez Benavides \cite{D-B} asserting that every reflexive space can be renormed to have fixed point property. This is in fact a meaningful improvement of a result of van Dust \cite[Theorem 1]{vD} which encompasses the class of all separable reflexive spaces. It should be stressed, however, that the question of whether reflexive spaces have FPP remains open (even for super-reflexive spaces). 


\subsection{Goal}\label{sec:2}
Within the context of unconditional basis, one of the best-known fixed point results, Lin's theorem \cite{Lin1}, states that if a Banach space $X$ has a $\mathcal{D}$-unconditional basis with $\mathcal{D}<\frac{1}{2} (\sqrt{33} - 3)$, then {\it every weakly compact convex subset} $C$  has the FPP (weak-FPP for short). In light of this result, we are naturally led to ask whether $\mathcal{D}<\frac{1}{2} (\sqrt{33} - 3)$ can be relaxed. Precisely:

\begin{itemize}
\item[($\mathcal{Q}1$)] Does every $X$ with an unconditional basis have the weak-FPP?\vskip .15cm
\item[($\mathcal{Q}2$)] Does every $X$ with a $1$-suppression unconditional basis have the weak-FPP?\vskip .15cm
\end{itemize}

One technical difficulty behind these issues is that subsequences of weakly null sequences can have large unconditional constants. Despite that, Lin \cite{Lin1} showed that under super-reflexivity ($\mathcal{Q}2$) has an affirmative answer. However, Lin's approach does not seem to be effective in general spaces with such bases. In line with these matters, the author \cite{Bar} gave a small step towards ($\mathcal{Q}1$) by proving the following improvement: Every Banach space with a shrinking $\mathcal{D}$-unconditional basis such that $\mathcal{D}<\sqrt{6}-1$ has the weak-FPP. Another related question is

\begin{itemize}
\item[($\mathcal{Q}3$)] Can $\co$ be equivalently renormed to have the FPP? 
\end{itemize}

Although one knows that $\co$ has a $1$-unconditional, up to now there is no partial positive answer and this has been in fact a difficult open problem in the theory. 

The goal of the present paper, which can be seen as a continuation of the work started in \cite{Bar}, is to take one further step towards the issues ($\mathcal{Q}1$), ($\mathcal{Q}2)$ and ($\mathcal{Q}3$), drawing attention to some facts that are implicitly contained in the literature and their respective links with the theory. Here we are mainly concerned with the FPP for the class of affine nonexpansive mappings. That is, nonexpansive maps $T\colon C\to C$ satisfying
\[
T(\lambda x+ (1- \lambda) y)= \lambda T(x) + (1-\lambda)T(y) \text{ for all } \lambda\in [0,1] \text{ and } x, y\in C.
\] 

In Section \ref{sec:3} we set up the notation, definitions and give some preliminary propositions. Our results are stated and proved in Section \ref{sec:4}. We finish the paper in Section \ref{sec:11} with some examples and further implications of our results. 


\section{Preliminaries}\label{sec:3}
Our notation is standard and follows \cite{AK, FHHMZ} for the most part. We will recall some notions for completeness' sake. $\co$ is the Banach space of real sequences tending to zero, under the supremum norm, and $\ell_1$ is the Banach space of absolutely summable real sequences with its usual $\ell_1$-norm. A sequence $(x_n)$ in a Banach space $X$ is called a {\it semi-normalized} if $0< \inf\| x_n\| \leq \sup\|x_n\|< \infty$. It is called {\it basic} if it is a Schauder basis for its closed linear span $[x_i]$. If $(x_n)$ and $(y_n)$ are basic sequences in Banach spaces $X$ and $Y$, respectively, we say that $(x_n)$ $C$-dominates $(y_n)$ for $C\geq 1$ (or that $(y_n)$ is $C$-dominated by $(x_n)$), denoted by $(y_n) \lesssim_C (x_n)$, if $\| \sum_{i=1}^n a_i y_i\| \leq C\| \sum_{i=1}^n a_i x_i\|$ for all scalars $(a_i)_{i=1}^n$. They are said to be equivalent if $(y_n) \lesssim_A (x_n)\lesssim_B (y_n)$. Thus, two basic sequences $(x_n)$ and $(y_n)$ are equivalent if $[x_n]$ and $[y_n]$ are linearly isomorphic. Clearly every basic sequence equivalent to the unit basis of $\co$ is semi-normalized and weakly null. For brevity we will call such sequences as {\it $\co$-sequences}. A basic sequence $(x_n)$ is called {\it spreading} (resp. {\it $1$-spreading}) if it is equivalent (resp. $1$-equivalent) to all of its subsequences. A Schauder basis $(e_i)$ is called {\it $\mathcal{D}$-unconditional} if $\| \sum_{i=1}^n \epsilon_i c_i e_i\| \leq \mathcal{D}\| \sum_{i=1}^n c_i e_i\|$ for all scalars $(c_i)_{i=1}^n\subset \mathbb{R}$ and $(\epsilon_i)_{i=1}^n\subset\{ -1,1\}$. Further, $(e_i)$ is called {\it $\beta$-suppression unconditional} if for all $n$, scalars $(c_i)_{i=1}^n$ and $F\subset \{ 1, \dots, n\}$, $\| \sum_{i\in F} c_i x_i\| \leq \beta \| \sum_{i=1}^n c_i e_i\|$. It is easily seen that if $(e_i)$ is $\mathcal{D}$-unconditional, then it is $\frac{1}{2}(\mathcal{D}+1)$-suppression unconditional. Also, if $(e_i)$ is $\beta$-suppression unconditional, it is $2\beta$-unconditional. A {\it conditional basis} is a Schauder basis that is not unconditional. A Schauder basis $(e_i)$ for a Banach space $X$ is called {\it boundedly complete} if for every sequence $(a_i)_{i=1}^\infty$ in $\mathbb{R}$ with $\sup_{N\geq 1}\big\| \sum_{i=1}^N a_i e_i\big\| < \infty$, the series $\sum_{i=1}^\infty a_i e_i$ converges in norm. If $\sup_{N\geq 1}\| \sum_{i=1}^N e_i\|< \infty$ then it is called {\it strongly bounded}. $(e_i)$ is called {\it shrinking} if $(e^*_i)$ (coefficient functionals) is a Schauder basis for $X^*$. For example, the unit basis of $\co$ is shrinking and $1$-unconditional, but not boundedly complete, while the unit basis of $\ell_1$ is $1$-unconditional, not shrinking and boundedly complete. It is well known (see \cite[Proposition 11.3.6]{AK}) that every semi-normalized weakly null $1$-spreading basic sequence is $1$-suppression unconditional. For $n\in\mathbb{N}$, let $P_n, R_n\colon X\to X$ denote the natural projections given by $P_n(x)=\sum_{j=1}^n e^*_j(x) e_j$ and $R_n:= I - P_n$, respectively. The basis $(e_i)$ is called pre-monotone ($\beta$-premonotone) if $\| R_{n}(x)\|\leq \| x\|$ ($\| R_{n}(x)\|\leq \beta \|x\|$) for all $x\in X$ and $n\in \mathbb{N}$ (cf. \cite{BP}). $(e_i)$ is called monotone (resp. bimonotone, strongly bimonotone) if $\mathcal{K}_m:=\sup_n\| P_n\|=1$ (resp. $\mathcal{K}_b:=\sup_{m<n}\| P_{[m,n]}\|=1$, $\mathcal{K}_{sb}:=\sup_{A\in[\mathbb{N}]^{<\omega}}\big( \| P_A\| \vee \| I- P_A\|\big)=1$). Here $[\mathbb{N}]^{<\omega}$ denotes the family of all finite subsets of $\mathbb{N}$. Thus every bimonotone basis is pre-monotone, and every $1$-suppression unconditional basis is bimonotone. Note that if $(e_i)$ is a (spreading) basis and $Q_m(x)=\sum_{i=m}^\infty e^*_i(x) e_i$, then $\vertiii{x}= \sup_{m\in\mathbb{N}} \|Q_{m}(x)\|$ is an equivalent norm on $X$ for which $(e_i)$ is pre-monotone (and spreading). If $(e_i)$ is conditional and spreading, then the summing functional defined by $\mathfrak{s}\big( \sum_i a_i e_i\big)=\sum_i a_i$ is well defined and bounded (cf. \cite[p. 1212]{AMS}). A sequence of non-zero vectors $(u_n)$ is called a block basic sequence of $(e_i)$, if there exist a sequence $(F_n)\subset \mathbb{N}]^{<\omega}$ with $\max F_n< \min F_{n+1}$ for all $n$, and scalars $(\lambda_n)$ with $\lambda_i\neq 0$ for all $i\in F_n$ and $n\in \mathbb{N}$ such that $u_n=\sum_{i\in F_n} \lambda_i e_i$ for all $n\in \mathbb{N}$. We then call $F_n$ as the {\it support} of $u_n$. 

We close this section with two propositions. The first one is easy and reads:

\begin{prop}\label{prop:1sec3} Let $X$ be a Banach space with a $\beta$-premonotone Schauder basis $(e_i)$. Then every block basic sequence $(u_i)$ of $(e_i)$ is $\beta$-premonotone. 
\end{prop}

The second one, whose proof is contained in \cite[Lemma 2.5]{F}, is the following.

\begin{prop}\label{prop:2sec3} Let $(x_n)$ be a semi-normalized basic sequence in the unit ball of a Banach space $X$. If $(x_n)$ dominates all of its subsequences, then there exist $C\geq 1$ and a subsequence $(x'_n)$ of $(x_n)$ such that $(x_n)$ $C$-dominates each subsequence of $(x'_n)$. 
\end{prop}

\section{Main results}\label{sec:4}

Let us begin with the following definition. 

\begin{dfn} Let $k\in\mathbb{N}$. We say that a $\co$-sequence $(x_n)$ in a Banach space $X$ has the $(N_k1)$ property, it there exists a sequence $\alpha=(\alpha_n)$ in $[0,1)$ converging to $1$ such that 
\begin{equation}\label{eqn:nk1p}
\Bigg\| \sum_{n=k}^\infty \alpha_n a_n x_n\Bigg\| \leq \Bigg\|\sum_{n=1}^\infty a_n x_n\Bigg\|
\end{equation}
for all sequence of scalars $(a_n)_{n=1}^\infty\in \co$. 
\end{dfn}

We also say that a Banach space $(X,\|\cdot\|)$ has the $(N_k1)$ property if it contains some $\co$-sequence enjoying the ($N_k$1) property for some $k\in\mathbb{N}$. 

\begin{remark}
Note that $(N_11)$ property coincides with the $(N1)$ property introduced by \'Alvaro, Cembranos and Mendoza \cite[Definition 1]{ACM}.
\end{remark}

\begin{thm} Every Banach space space with the $(N_k1)$ property fails the fixed point property for affine nonexpansive maps. 
\end{thm}

\begin{proof} Let $K$ be the set considered in \cite[Theorem 2]{ACM} and define the mapping $T\colon K\to K$ by $T(\sum_{n=1}^\infty t_n x_n) = \sum_{n=k}^\infty \alpha_n t_n x_n  + \sum_{n=k}^\infty \beta_n x_n$, where $\beta_n = 1-\alpha_n$, $n\in\mathbb{N}$, and $(\alpha_n)$ is as in (\ref{eqn:nk1p}). Then $T$ is nonexpansive, affine and fixed point free. 
\end{proof}

In virtue of this we can consider the following weak notion of premonotonicity. 

\begin{dfn} Let $X$ be a Banach space. We say that a basic sequence $(x_n)$ in $X$ is $\mathfrak{t}$-premonotone if there exists an integer $k\geq 2$ so that $\| Q_{n}\|=1$ for all $n\geq k$, where $Q_n x = \sum_{i=n}^\infty a_i x_i$ for any $x= \sum_{i=1}^\infty a_i x_i\in [x_n]$. 
\end{dfn}

It is clear that every pre-monotone basic sequence is $\mathfrak{t}$-premonotone. 

\begin{prop}\label{prop:A} Let $X$ be a Banach space. Assume that $(x_n)$ is a $\mathfrak{t}$-premonotone basic sequence in $X$. Then there exists an integer $k\geq 2$ so that $(x_n)$ has the $(N_k1)$ property. 
\end{prop}

\begin{proof}  This follows directly from the $\mathfrak{t}$-premonotonicity assumption and the fact that for any scalars $(a_i)_{i=1}^m\subset\mathbb{R}$ we may write
\[
\sum_{n=k}^m \alpha_n a_n x_n= \alpha_k R_{k-1}\Bigg( \sum_{i=1}^m a_i x_i\Bigg) + \sum_{i=k}^{m-1}\big( \alpha_{i+1} - \alpha_i\big) R_i \Bigg( \sum_{n=1}^m a_n x_n\Bigg).
\]
\end{proof}

\begin{cor}\label{cor:6sec3} (i) If $(x_n)$ is a $\mathfrak{t}$-premonotone $\co$-sequence in a Banach space $X$, then it has the $(N_k1)$ property. (ii) Every Banach space containing a $\mathfrak{t}$-premonotone $\co$-sequence fails the fixed point property for affine nonexpansive mappings. 
\end{cor}

\begin{remark} We do not know whether there is a $\co$-sequence that does not satisfy $(N1)$ property but does satisfy $( N_k1)$ for some $k\geq 2$. However the renorming of $\co$ given by
\[
\| x\|= \sup_{i\in\mathbb{N}}|a_i - \frac{1}{3}a_1| \vee \sup_{i\geq 3}|a_i|
\]
satisfies $\| Q_m x\|=1$ for all $m\geq 3$ and $4/3\leq \|Q_2\|\leq 3/2$. 
\end{remark}

We are heading now towards getting sufficient conditions for the existence of pre-monotone $\co$-sequences. 

\begin{thm}\label{thm:B} Let $X$ be a Banach space having a pre-monotone Schauder basis $(e_i)$ that is not boundedly complete. Then $X$ contains a pre-monotone $\co$-sequence provided that one of the following conditions is verified:
\begin{itemize}
\item[(i)] The basis is unconditional.
\item[(ii)] The basis is spreading and $X$ does not embed into a space with an unconditional basis.
\item[(iii)] The basis is shrinking and is such that each of its strongly bounded block basic sequences dominate its subsequences. 
\end{itemize}
\end{thm}

\begin{proof} The proof of (i) is a direct consequence of Proposition \ref{prop:1sec3} and \cite[Theorem 3.3.2]{AK}. 

\vskip .04cm 

\noindent (ii).  Since $(e_i)$ is not boundedly complete, there exists $x^{**}\in X^{**}$ so that the series $\sum_{i=1}^\infty x^{**}(e^*_i)e_i$ does not converge in norm. As $X$ does not embed into a space with an unconditional basis, this clearly implies that $(e_i)$ is conditional. In addition, it is also strongly summing (cf. \cite[Proposition 6.5]{AMS}). It follows in particular that $\sum_{i=1}^\infty x^{**}(e^*_i)$ is convergent. Thus we may choose a strictly increasing sequence $(n_i)_i\subset \mathbb{N}$ so that if $x_i= \sum_{j=n_{i-1}+1}^{n_i} x^{**}(e^*_j) e_j$, then $(x_i)$ is semi-normalized and $\sum_{i=1}^\infty | \mathfrak{s}(x_i)|<\infty$, where $\mathfrak{s}$ is the summing functional. By the proof of \cite[Lemma 6.3]{AMS}, $(x_i)$ is equivalent to the unit vector basis of $\co$. Since $(e_i)$ is pre-monotone, so is $(x_i)$ and the result follows. 


\vskip .04cm 

\noindent (iii). Since $(e_i)$ is shrinking and not boundedly complete, it has a block basic sequence $(x_n)$ which is semi-normalized, weakly null and strongly bounded. According to an unpublished result of Johnson (cf. \cite[Theorem 4.1]{O}), if every subsequence of a normalized weakly null sequence $(u_n)$ has a further subsequence which is strongly bounded, then there exists a subsequence of $(u_n)$ equivalent to the unit basis of $\co$. Let us stress that the same proof given in \cite{O} works for semi-normalized weakly null sequences. Let $C\geq 1$ and $(x'_n)$ be as in Proposition \ref{prop:2sec3}. Then $(x'_n)$ is semi-normalized and weakly null. In addition, every subsequence of any subsequence of $(x'_n)$ is $C$-dominated by $(x_n)$. Since $(x_n)$ is strongly bounded, by Johnson's result  $(x'_n)$ has a subsequence $(x''_n)$ equivalent to the unit basis of $\co$. Since $(e_i)$ is pre-monotone, so is $(x''_n)$ as desired.  
\end{proof}

Our next result connects the FPP to PCP. Let's recall that a closed, bounded convex set $C$ in a Banach space $X$ is said to have the {\it point of continuity property} (PCP) (resp. {\it convex point of continuity property} (CPCP)) if for every non-empty norm-closed (convex) subset $K$ of $C$, the identity mapping $I$ on $K$ is weak-to-norm continuous at some point of $K$. The space $X$ is said to have the PCP (resp. CPCP) if its closed unit ball $B_X$ has the PCP (resp. CPCP). More details on PCP are included in Remark \ref{rmk:PCP}. 

\begin{thm}\label{thm:D} Let $X$ be a Banach space with a $1$-suppression unconditional basis. Assume that $X$ fails to have the $(N_k1)$ property. Then $X$ has PCP. 
\end{thm}

\begin{proof} Suppose the thesis is false. Then the definition of PCP in conjunction with \cite[Lemma 2.1]{PA} yields a non-empty closed subset $K$ of $B_X$ so that, for some $\delta>0$, every weak-neighborhood of $K$ has diameter at least $\delta$. Let $(e_i)$ denote the basis of $X$. We will now mimic the arguments from \cite[Proposition 2.3]{AORos} to build a semi-normalized block basis of $(e_i)$ dominated by the unit basis of $\co$. For $n\in\mathbb{N}$, consider the natural projections $P_n$, $R_n$. Note that $\diam(K)\geq \delta$, so there is a point $y_1\in K$ with $\| y_1\| \geq \delta/2$. Pick $m_1\in\mathbb{N}$ so that $\| R_{m_1}y_1\|< \delta/2$ and define $z_1 = P_{m_1}y_1$. Let $M_1=\sup_{1\leq i\leq m_1}\| e_i\|$ and  
\[
V_1=\Big\{ x\in K\colon \sum_{j=1}^{m_1}| e^*_j( x -y_1 )| < \delta/2^{m_1+2}M_1\Big\}.
\]
Then $\diam(V_1)\geq \delta$ and hence there is $y_2\in V_1$ so that $\| y_2 - y_1\|\geq \delta/2$. Set $x_1= y_1$ and $x_2= y_2 - y_1$. Pick $m_2> m_1$ so that $\|R_{m_2} x_2\| < \delta/2^{3}$ and define $z_2 = P_{(m_1, m_2]} x_2$. It follows that $z_2\neq 0$ and $\| x_2 -z_2\|< \delta/2^2$. By induction, assume that $n\geq 1$, $m_{n+1}> m_n$, vectors $\{x_i\}_{i=1}^{n+1}$ and non-null block vectors $\{z_i\}_{i=1}^n$ have been already built so that $\sum_{i=1}^{n+1} x_i\in K$, $\| x_k\|\geq \delta/2$ and $\| x_k - z_k\|< \delta/2^k$ for all $k=1,\dots, n$. Put $M_{n+1}=\max_{1\leq j\leq m_{n+1}}\|e_j\|$ and $y_{n+1}= \sum_{i=1}^{n+1} x_i$. Set
\[
V_{n+1}=\Big\{ x\in K \colon \sum_{j=1}^{m_{n+1}} |e^*_j(x - y_{n+1} )| < \delta/2^{m_{n+1}+2}M_{n+1}\Big\}.
\]
Since $\diam(V_{n+1})\geq \delta$ we can find $y_{n+2}\in V_{n+1}$ so that $\| y_{n+2} - y_{n+1}\|\geq \delta/2$. Set $x_{n+2}=y_{n+2} - y_{n+1}$. Now pick $m_{n+2}> m_{n+1}$ so that $\| R_{m_{n+2}} x_{n+2}\| < \delta/ 2^{n+3}$. Define $z_{n+2}= P_{(m_{n+1}, m_{n+2}]} x_{n+2}$. Then $\| x_{n+2} - z_{n+2}\|\leq \delta/2^{n+2}$. This concludes the induction process. 

By construction, $(z_n)$ is a semi-normalized block basic sequence of $(e_i)$. Let us show that $(z_n)$ is dominated by the unit basis of $\co$. Fix any scalars $(a_i)_{i=1}^m$. Next pick a functional $\varphi\in S_{X^*}$ so that $\| \sum_{i=1}^m a_i z_i\|= \varphi( \sum_{i=1}^m a_i z_i)$. We can regroup terms and write $\sum_{i=1}^m a_iz_i= \sum_{k=1}^m a_{i_k} z_{i_k}$ with $a_{i_1}\geq a_{i_2}\geq \dots \geq a_{i_m}$. Consequently, 
\[
\begin{split}
\Bigg\|\sum_{i=1}^ma_i z_i\Bigg\| &= \sum_{k=1}^{m-1}( a_{i_k} - a_{i_{k+1}}) \varphi\Big(\sum_{s=1}^i z_{i_s}\Big) + a_{i_m} \varphi\Big( \sum_{s=1}^m z_{i_s}\Big)\\[1.2mm]
&\leq \sum_{k=1}^{m-1}(a_{i_k} - a_{i_{k+1}})\Bigg\| \sum_{s=1}^iz_{i_s}\Bigg\|+ |a_{i_m}|\Bigg\|\sum_{s=1}^m z_{i_s}\Bigg\|\\[1.2mm]
&\leq \Bigg(\sum_{k=1}^{m-1}(a_{i_k} - a_{i_{k+1}}) + |a_{i_m}|\Bigg)\Bigg\| \sum_{i=1}^mz_i\Bigg\|\leq 3\Bigg\| \sum_{i=1}^m z_i\Bigg\|\sup_{1\leq i\leq m}|a_i|,
\end{split}
\]
where in the second inequality we used that $(z_i)$ is $1$-suppression unconditional. Since
\[
\begin{split}
\Bigg\| \sum_{i=1}^m z_i\Bigg\|& \leq \sum_{i=1}^m \|z_i - x_i\| + \Bigg\| \sum_{i=1}^m x_i\Bigg\|\leq \delta + \sup_{x\in K}\| x\|,
\end{split}
\]
we deduce that $(z_n)$ is a pre-monotone $\co$-sequence. By Corollary \ref{cor:6sec3}, $X$ has the $(N1)$ property, contradicting the hypothesis of the theorem. 
\end{proof}


Regarding the weak-FPP, we obtain the following  results.

\begin{thm}\label{thm:F} Let $X$ be a Banach space having a $\mathcal{D}$-unconditional basis with $\mathcal{D}<2$. Assume that $Z$ is a Banach space with a strongly bimonotone basis and $\dist_{BM}(X,Y)=1$ for some subspace $Y$ of $Z$. Then $X$ has the weak fixed point property. 
\end{thm}

\begin{proof}
Let $(e_i)$ denote the basis of $X$. Assume for a contradiction that $X$ fails the weak-FPP. Let $K$ be a weakly compact convex subset of $X$ which is minimal-invariant for a fixed point free nonexpansive mapping $T\colon K\to K$. Take $(x_n)$ an approximate fixed point sequence of $T$. By passing to a subsequence and taking suitable scalings, we may assume that $\diam(K)=1$, $0\in K$ and $(x_n)$ is semi-normalized and weakly null. Let $\varepsilon>0$ be fixed to be chosen later. Since $(e_i)$ is $
\mathcal{D}$-unconditional, after passing to a further subsequence, we may also assume that $(x_n)$ is $(\mathcal{D} +\varepsilon)$-unconditional. 

By assumption there is a Banach space $Z$ having a strongly bimonotone basis such that $\dist_{BM}(X,Y)=1$ for some subspace $Y$ of $Z$. Let $J\colon X\to Y$ be an isomorphism satisfying $\|J^{-1}\|\leq 1$ and $\|J\|\leq 1 +\varepsilon$. Following \cite{Bar} we now define a new norm on $Z$ by
\[
|z|=\inf\big\{ \|x\|_X + \|z- Jx\|_Z\colon x\in X\big\}.
\]
It follows that $J$ is an isometry from $X$ into $(Z, |\cdot|)$. Moreover, one has
\begin{equation}\label{eqn:1sec9}
\frac{1}{1 +\varepsilon}\|z\|_Z\leq |z|\leq \|z\|_Z\quad\text{for all }z\in Z.
\end{equation}

Let $(z_i)$ be the basis of $Z$ and let $(P_i)$ denote its canonical basis projections. Inequality (\ref{eqn:1sec9}) readily shows that $\mathcal{K}^{|\cdot|}_{sb}$, the  strong bimonotonicity projection constant of $(z_i)$ with respect to the norm $|\cdot|$, is not larger than $1+\varepsilon$. For $n\in \mathbb{N}$ let $y_n=J(x_n)$. Since $J$ is an isometry, $(y_n)$ is weakly null and $(\mathcal{D}+\varepsilon)$-unconditional in $(Z,|\cdot|)$. On the other hand, applying the gliding hump method, we obtain a subsequence $(y_{n_i})$ of $(y_n)$ and an increasing sequence of consecutive intervals $(F_i)$ of integers in $\mathbb{N}$ such that 
\begin{equation}\label{eqn:2sec9}
\lim_{i\to\infty}| y_{n_i} - P_{F_i} y_{n_i}|=0.
\end{equation}

Let now $[X]$ and $[Z]$ denote the ultrapowers of $X$ and $Z$ respectively. Define $K_J:=J(K)$ and $T_J\colon K_J\to K_J$ by $T_J(Jx)= J(Tx)$, for all $x\in K$. Next  set 
\[
[\mathcal{M}]=\left\{  [v_i] \in[K_J] \colon\, \begin{matrix}&\exists\, x\in K\text{ such that }\, \big| [v_i] - Jx\big|\leq(\mathcal{D}+\varepsilon)/2, \text{ and }\\
&| [v_i] -[y_{n_i}]|\vee \big| [v_i] -[y_{n_{i+1}}]\big|\leq 1/2.
\end{matrix}\right\}.
\]
Clearly $[\mathcal{M}]$ is a nonempty, closed convex subset of $[K_J]$ which is $[T_J]$-invariant. By Lin's lemma \cite{Lin1}, $\sup\{ \big| [v_i]\big| \colon [v_i] \in [\mathcal{M}]\}=1$. Recall that $\diam(K)=1$ and $J$ is an isometry. As it turns out however, for $\varepsilon$ small enough, one has
\[
\begin{split}
\big| [v_i]\big|&\leq \frac{1}{2}\Big( \big|[P_{F_i}] + [P_{F_{i+1}}]\big| \big| [v_i] - Jx\big| + \big|[I] - [P_{F_i}]\big| \big| [v_i] - [ y_{n_i}]\big| + \big| [ I] - [P_{F_{i+1}}]\big| \big| [v_i] - [y_{n_{i+1}}]\big|\Big)\\[1.2mm]
&\leq \frac{1}{2}\Big( (1 +\varepsilon)\frac{\mathcal{D} +\varepsilon}{2} + \frac{1+\varepsilon}{2}+ \frac{1+\varepsilon}{2}\Big)= \frac{1+\varepsilon}{2}\Big(\frac{\mathcal{D} +\varepsilon}{2} +1\Big)<1.
\end{split}
\]
This contradiction proves the theorem. 
\end{proof}

\begin{thm}\label{thm:G} Let $X$ be a Banach space with a Schauder basis $(e_i)$. Assume that every block basis sequence has a subsequence equivalent to the basis. Then $X$ has the weak fixed point property in each of the following two situations:
\begin{itemize}
\item[(i)] $(e_i)$ is a $1$-suppression unconditional basis. 
\item[(ii)] $(e_i)$ is  a $1$-spreading basis. 
\end{itemize}
\end{thm}

\begin{proof}
It is not hard to see (after applying an infinitary analog of Ramsey's theorem) that the block basis assumption implies that $(e_i)$ is spreading (cf. \cite[p.397]{FPR}). 

\vskip .05cm 
\noindent (i).  By James's theorem either $X$ is reflexive or contains a subspace isomorphic to $\co$ or $\ell_1$. Since Schur spaces have weak-FPP, by the assumption we may assume that $X$ does not contain copies of $\ell_1$. Suppose by contradiction that $X$ fails the weak-FPP. Then $X$ contains a semi-normalized weakly null basic sequence $(x_n)$ which generates an $\ell_1$-spreading model (cf. proof of \cite[Theorem (3)]{GFal}). Since no spreading model of $\co$ is isomorphic to $\ell_1$, $X$ is reflexive. Hence $X$ is reflexive and has an $\ell_1$-spreading model generated by $(x_n)$. This fact (i.e., $(x_n)$ generates an $\ell_1$-spreading model) remains true under any equivalent renorming. It turns out that (up to renorming) any spreading model of $X$ is equivalent to $(e_i)$ (cf. \cite[Lemma 3]{FPR}). But this, however, implies $X$ is isomorphic to $\ell_1$, contradiction.  

\vskip .05cm 
\noindent(ii). By Rosenthal's $\ell_1$-theorem, either $(e_i)$ is equivalent to the unit basis of $\ell_1$ or $(e_i)$ is weak Cauchy. Assume without loss of generality that $(e_i)$ is weak Cauchy. Now, we analyse two cases. If $(e_i)$ is weakly null, then $(e_i)$ is $1$-suppression unconditional and by the previous result $X$ has the weak-FPP. Assume now that $(e_i)$ is not weakly null. For $n\in\mathbb{N}$, set $d_n= e_{2n-1} - e_{2n}$. Then $(d_n)$ is a weakly null and $1$-spreading (cf. \cite[Proposition 11.3.6]{AK}). Since $(d_n)$ defines a basis for $X$ with such properties, by the previous case, the result follows.   
\end{proof}


\medskip 

\section{Consequences and final considerations}\label{sec:11}
As far as we know, the results presented in this work are not available in the published literature. Although the tools used in the demonstrations are not new, we believe that their use in the income statements not only represents something relatively new, but also illustrates how difficult and demanding the problems considered here are. It is important to emphasize that it has been a recurrent strategy in many works in metric fixed point theory, the use of non-elementary techniques (albeit well established) from Banach space theory in the verification of both the FPP and the failure of FPP. Our results go in that direction. We conclude this work with some additional results and remarks.

\begin{prop}\label{prop:1sec10} Let $X$ be a Banach space with a pre-monotone basis. If $X$ contains an isomorphic copy of $\co$, then it fails the FPP for affine nonexpansive mappings.
\end{prop}

\begin{proof} Let $(y_i)$ be a semi-normalized weakly null sequence in $X$ equivalent to the unit basis of $\co$. By the gliding hump method, $(y_i)$ admits a subsequence which is equivalent to a block basic sequence $(x_i)$ of the basis of $X$. Thus $(x_i)$ is a pre-monotone $\co$-sequence. By Corollary \ref{cor:6sec3}, the result follows. 
\end{proof}

\begin{rmk} It is natural to ask if every Banach space containing a copy of $\co$ must contain a pre-monotone $\co$-sequence. Dowling, Lennard and Turett \cite{DLT3} proved that if a Banach space $X$ contains a copy of $\co$ then it contains an asymptotically pre-monotone $\co$-sequence, that is, a sequence $(x_n)$ which is equivalent to the unit basis of $\co$ and for some decreasing null sequence $(\delta_n)$ in $(0,1)$ one has $\|R_n\| \leq (1 + \delta_n)$ for all $n\in \mathbb{N}$. 
\end{rmk}

\begin{prop}\label{prop:2sec10} Let $X$ be a Banach space with a Schauder basis in which every block basic sequence has a block basic sequence which is not boundedly complete. Assume that $X$ contains a pre-monotone basic sequence. Then $X$ contains a pre-monotone basic sequence which is not boundedly complete. 
\end{prop}

\begin{proof} Let $(e_i)$ denote the basis of $X$ and let $(y_n)$ be a pre-monotone basic sequence in $X$. By a result due to Dean, Singer and Sternbach (see \cite[Proposition 3]{DSS}), $(y_n)$ has a block basic sequence $(x_n)$ which is equivalent to a block basic sequence with respect to $(e_i)$. By assumption, $(x_n)$ is not boundedly complete and this concludes the proof. 
\end{proof}

\begin{rmk} Examples of spaces having a basis as in the statement of Proposition \ref{prop:2sec10} include those that are $\co$-saturated. For $p\geq 1$, the $p^{\textrm{th}}$ James-Schreier space $V_p$ defined in \cite[Theorem 5.2]{BL} is an important example of a $\co$-saturated space. In fact, the standard basis $V_p$ is monotone shrinking and does not embed into a Banach space with an unconditional basis. Recall that for $X$ and $Y$ infinite-dimensional Banach spaces, one says that $X$ is $Y$-{\it saturated} if each closed, infinite-dimensional subspace of $X$ contains a subspace which is isomorphic to $Y$. We were not been able to verify whether or not $V_p$ fails the FPP. 
\end{rmk}

In what follows we shall consider some consequences of Theorem \ref{thm:B}.

\begin{cor}\label{cor:A} Let $X$ be a Banach space with a pre-monotone unconditional basis. If no subspace of $X$ is isomorphic to $\ell_1$ then the following are equivalent:
\begin{itemize}
\item[(i)] $X$ is reflexive.
\item[(ii)] $X$ has the fixed point property for affine nonexpansive mappings. 
\item[(iii)] Every closed convex subset of $B_X$ has the FPP for affine nonexpansive mappings. 
\end{itemize}
\end{cor}

\begin{proof} (i) $\Rightarrow$ (ii) follows from the Schauder-Tychonoff's fixed point theorem. Clearly (ii) implies (iii). Let's prove  (iii) $\Rightarrow$ (i). Assume that $X$ is not reflexive.  Taking into account that $X$ does not contain a copy of $\ell_1$, from \cite[Theorem 1 and Lemma 2]{J1} we see that $(e_i)$ cannot be boundedly complete. By the proof of Theorem \ref{thm:B} (i), the result follows. 
\end{proof}

\begin{cor}\label{cor:B} Let $X$ be a Banach space with a pre-monotone unconditional basis. If no subspace of $X$ is isomorphic to $\ell_1$, then $X$ fails the FPP for affine nonexpansive mappings if and only if $X$ contains a basic sequence which is not boundedly complete. 
\end{cor}

\begin{proof} It suffices to apply a sequential characterization of reflexivity due to Singer \cite[Theorem 2 and  Corollary 1]{S} (see also \cite[Proposition]{Z}) in concert with Corollary \ref{cor:A}. 
\end{proof}

\begin{cor}\label{cor:C} Let $X$ be a Banach space with a $1$-unconditional basis $(e_i)$. Assume that $X$ does not contain $\ell_1$ isomorphically. Then the following assertions are equivalent:
\begin{itemize}
\item[(i)] $X$ is reflexive.
\item[(ii)] $X$ has the fixed point property for nonexpansive mappings. 
\item[(iii)] Every closed convex subset of $B_X$ has the FPP for nonexpansive mappings. 
\end{itemize}
\end{cor}

\begin{proof} (i) $\Rightarrow$ (ii) follows directly from Lin's theorem \cite{Lin1}, since every closed, bounded convex subset of a reflexive space is weakly compact. Clearly (ii) implies (iii). Finally, from Corollary \ref{cor:A} we see that (iii) implies (i).
\end{proof}

\begin{rmk} It is noteworthy to point out that Lin's $\ell_1$-renorming \cite{Lin2} shows that the assumption on the absence of isomorphic copies of $\ell_1$ cannot be removed from the above statements. Likewise, the assumption of not being boundedly complete is necessary in our main theorem. Our results are directly related to those in \cite{B-MJ2}. In particular, the above corollaries should be compared with \cite[Corollary 5.4]{B-MJ2}.
\end{rmk}

\begin{ex}[Schreier's space] The classical Schreier space $X_{\mathcal{S}}$ is the completion of $\coo$ with respect to the following norm
\[
\| x \|_\mathcal{S}= \sup_{E\in \mathcal{S}}\sum_{i\in E}| a_i|,\quad x=(a_i)_{i=1}^\infty\in \coo,
\]
where $\mathcal{S}=\big\{ E\in [\mathbb{N}]^{<\infty} \colon |E|\leq \min E\big\}$ is the family of so-called {\it admissible} sets. It is known that the unit basis of $X_{\mathcal{S}}$ is $1$-unconditional and not boundedly complete. In fact, one can prove that $X_{\mathcal{S}}$ contains an isometric copy of $\co$ (cf. \cite[Theorem 0.5, Proposition 0.7 and Corollary 0.8]{CS}). 
\end{ex}

\begin{rmk} Notice that the assertion (i) of Theorem \ref{thm:B} shows in particular that any renorming of $\co$ enjoying the fixed point property must be such that its unit basis cannot be pre-monotone. This fact is slightly more general than Proposition 8 in \cite{ACM}. 
\end{rmk}

\begin{ex}[A non-$1$-unconditional $\co$-sequence that is $1$-suppression unconditional] Let us consider the following norm in $\co$:
\[
\| (a_i)\|= \sup_{i, j\in \mathbb{N}}|a_i - a_j|,\quad (a_i)_{i=1}^\infty\in \co.
\]
As it was pointed out in \cite[Example 10]{ACM}, the unit basis of $\co$ is not $1$-unconditional with respect to $\|\cdot\|$. However, one easily shows that it is $1$-suppression unconditional. 
\end{ex}

\begin{rmk}  A Banach space $X$ as in Theorem \ref{thm:B} (iii) may fail to contain $\co$-sequences without the extra assumption that strongly bounded block basic sequences of the basis must dominate their subsequences. In fact, recall that the classical James's space $J_2=\big\{ (a_i)_{i=1}^\infty \in c_0 \,\big| \| (a_i)_{i=1}^\infty\|_{J_2}<\infty\big\}$ where $\| \cdot\|_{J_2}$ is defined by 
\[
\| (a_i)_{i=1}^\infty\|_{J_2}= \sup\Bigg\{ \Big( \sum_{k=1}^l | a_{p_{k+1}} - a_{p_k}|^2\Big)^{1/2}\,\colon \, 2\leq l,\, 1\leq p_1< \dots < p_l\Bigg\},
\]
is not reflexive, does not contain any subspaces isomorphic to $\co$ or $\ell_1$ and its unit basis is non-unconditional shrinking and monotone. It is worth mentioning that $J_2$ has the weak-FPP (cf. \cite{K}). 
\end{rmk}

\begin{rmk} It is known that for infinite compact Hausdorff spaces $K$, $C(K)$ does not have a $1$-suppression unconditional basis (see e.g. \cite[Remark 2.13]{ALMT}). On the other hand, it was pointed out in \cite{B-MJ1} that if $C(K)$ is isomorphic to $\co$ with $K$ metrizable, then it has the weak-FPP. However it seems to be an open question whether or not $C(K)$ has this property if $K$ is a scattered set such that $K^{(\omega)}\neq \emptyset$. 
\end{rmk}

\begin{rmk} It is worth mentioning that complementation properties of isomorphic copies of $\co$ has been useful in metric fixed point theory (cf. \cite{B-MJ2}). Cembranos \cite{C} proved that if $E$ is an infinite dimensional Banach space, then the space $C(K, E)$  of all continuous $E$-valued functions on $K$ contains a complemented copy of $\co$. In fact, her proof shows slightly more than this; it shows that $C(K,E)$ contains a complemented $1$-suppression unconditional basic sequence equivalent to the unit basis of $\co$. 
\end{rmk}

\begin{rmk}\label{rmk:PCP} Historically PCP seems to have first studied by Namioka \cite{N}. CPCP was introduced by Bourgain \cite{B}. Clearly PCP implies CPCP. It is well-known that Banach spaces with Radon-Nikod\'ym property (RNP), including separable dual spaces, have the CPCP. It is known that PCP is separably determined \cite{Bou}. Rosenthal \cite{Ros} showed that every semi-normalized basic sequence in a Banach space with PCP has a boundedly complete subsequence. It is also known that if a Banach space $X$ has an unconditional basis and fails the CPCP, then $X$ contains a $\co$-sequence (cf. \cite[Corollary 2.11]{Sch}). So, Theorem \ref{thm:D} provides a more general fact.  The Banach space $B_\infty$ constructed in \cite{GM} satisfies the CPCP, fails PCP, has a semi-normalized supershrinking Schauder basis and does not contain subspaces isomorphic to $\co$. 
\end{rmk}

\begin{rmk} The arguments used in Theorem \ref{thm:G} can also be applied to show that if $X$ has weak Banach-Saks property and $\dist_{BM}(X, Y)=1$, where $Y$ is a subspace of a Banach space $Z$ having a strongly bimonotone basis, then $X$ has the weak fixed point property. This yields a slightly generalized version of a known result of Garc\'ia-Falset \cite{GFal}. 
\end{rmk}

The following notion was introduced by C. Lennard and Nezir in \cite{LN}, where important characterizations of reflexivity in terms of the FPP were established. 

\begin{dfn}[Cascading Nonexpansive Mappings] Let $X$ be a Banach space and $C$ be a closed convex subset of $X$. Let $T\colon C\to C$ be a mapping and define $C_0=C$, $C_n=\overline{\conv}\hskip .05cm T(C_{n-1})$ for every $n\in \mathbb{N}$. The mapping $T$ is said to be cascading nonexpansive if there exists a sequence $(\lambda_n)_{n\geq 0}\subset [1,\infty)$ with $\lim_n \lambda_n=1$ and such that $\| Tx - Ty\| \leq \lambda_n\| x - y\|$ for all $x, y\in C_n$ and for all $n\geq 0$. 
\end{dfn}

Finally, we refer the reader to Benavides and Jap\'on Pineda \cite{B-MJ2} for further fixed-point characterizations of weak compactness in spaces with unconditional basis for the class of cascading nonexpansive maps. As far as we are concerned, the methods used in these last works do not allow us to obtain our results by considering the class of nonexpansive mappings (even in the special case where the basis is $1$-unconditional).

\medskip 

{\bf Conflicts of Interest.} The author declare that there are no competing interest. 

\smallskip 
 
{\bf Acknowledgements.} This work was started when the author visited the Mathematics Department of the Federal University of Amazonas (UFAM), from December 16, 2020 to January 7, 2021. For this reason, the authors also thank Professors Fl\'avia Morgana and Jeremias Le\~ao for the kind invitation and support. The author also wish to thank Professor Bruno M. Braga for helpful comments. 

The author especially thanks the reviewers for their thorough reading of the article, pointing out some gaps, errors and needs for improvement that culminated in this final version of the article.

\end{document}